\documentclass[11pt]{article}
\usepackage{hyperref}
\usepackage[english]{babel}

\usepackage{graphicx}
\usepackage{amssymb}
\usepackage{amsthm, amsmath,amssymb}
\usepackage[all]{xy}
\usepackage{pstricks,pstricks-add,pst-node,pst-text,pst-3d}
\usepackage{graphicx}
\usepackage{psfrag}
\usepackage{multirow}

\usepackage[applemac]{inputenc}

\newcommand{\Rb}{\mathbb{R}}
\newcommand{\Qb}{\mathbb{Q}}

\newcommand{\mathsym}[1]{{}}
\newcommand{\unicode}[1]{{}}

\newtheorem{theorem}{Theorem}[section]
\newtheorem{corollary}{Corollary}

\newtheorem{proposition}{Proposition}

\theoremstyle{definition}
\newtheorem{definition}[theorem]{Definition}
\newtheorem{remark}{Remark}

\newtheorem{example}{Example}

\begin{document}

\title{High degree quadrature rules with pseudorandom rational nodes}

\author{M{\' a}rio M. Gra{\c c}a
\thanks{Departamento de Matem\'{a}tica,
Instituto Superior T\'ecnico, Universidade  de Lisboa,
 Av. Rovisco Pais,                 
1049--001 Lisboa, Portugal, 
e-mail: \texttt{\url{mario.meireles.graca@tecnico.ulisboa.pt}}}
}
 
\maketitle

\begin{abstract}

\noindent
After introducing the definitions 
of {\em positive, negative} and {\em companion rules},  from a given  pair of companion rules we construct a new rule  with higher degree of precision  The scheme is generalized giving rise to a transformation which we call the {\em mean rule}. 
We show that the mean rule is the best approximation, in the sense of least\--squares,
obtained from a linear combination of two rules of the same degree of precision. Finally, we show that a rule of degree $2k+1$ can be constructed as linear combination of $k+1$ rules of degree one and rational pseudorandom nodes. Several worked examples are presented.
 \end{abstract}



 \section{Introduction}\label{introd}

 \noindent
 Given two quadrature rules $A(g)$ and $B(g)$ with the same degree of precision $m\geq 0$, we begin by proposing a scheme to construct a new rule $Y(g)$ of higher degree.  One desirable assumption is that the rules $A(g)$ and $B(g)$ are {\em companion}, in the sense that one can assign opposite signals to the respective error. In particular, we show that the basic quadrature rules known as midpoint, trapezoidal, and Simpson rules can all be obtained as linear combinations of companion rules of lower degree of precision. The theoretical background applied in this work relies on the method of undetermined coefficients (\cite{dahlquist}, p. 565), (\cite{gautschi} p. 170).

 \medskip
 \noindent
We generalize the referred scheme by presenting a rule transformation $W(g)$ (defined in the set ${\cal Q}$ of rules of degree $m$), which to a pair of rules $\left(A(g),B(g)\right)$ of ${\cal Q}$, assigns a new rule of greater degree. This leads  to an algorithm to obtain quadrature rules of arbitrary order of precision, as suggested in the worked examples. The rule $W(g)$ is called the {mean rule} since it is a weighted mean of $A(g)$ and $B(g)$. It can be seen as a least\--squares best approximation as discussed in paragraph \ref{best}.

\medskip
\noindent
The main results of this paper are discussed in Section \ref{openrules}. We first show that if one takes a set of $k+1$ of open rules of degree one, $Q_0(g),\ldots, Q_k(g)$, whose first member is the midpoint rule $Q_0(g)=2\, g(0)$ and the other members have two symmetrical rational nodes,  there exists a unique  linear combination of the rules such that the combined rule $W_k(g)$ has degree $2\,k+1$ (see Proposition \ref{prop3}). As an illustration we apply the composite version of the  rule $W_5 (g)$ to obtain approximations of $\pi$ with 60 significant digits (Example \ref{exemplo9}) using the model function $g(t)=2/(1+t^2), \,\-1\leq t\leq 1$.

\medskip
\noindent
Finally we show that one can expand the scheme considering combined rules where the nodes of the $1$\--degree starting rules are pseudorandom rational numbers (paragraph \ref{secran}). In particular, we use a pseudorandom $151$\--degree combined rule giving an approximation of $\pi$ with more than $500$ significant digits as detailed in Example \ref{exemplo10}.

\section{Notation and definitions}\label{taylor}


\noindent
A quadrature rule $Q(f)$ is an approximation of the integral $\displaystyle \int_a^b f(x) dx$ obtained using values of $f$ (and/or its derivatives) on a discrete set of points in $[a,b]$ (see for instance  \cite{engels}).

\medskip
\noindent
Without loss of generality, we consider $[-1,1]$ to be  the interval of integration and we will denote by $Q(g)$ a general quadrature rule to approximate the integral $I(g)=\displaystyle \int_{-1}^1 g(t)\, dt$. 

\medskip
\noindent
 Note that a rule $Q(g)$, defined in $[-1,1]$, can be rewritten for the interval $[a,b]$ as $Q(f)$, through a change of variable defined by the bijection
$$
\sigma(t)= a+ \displaystyle \frac{b-a}{2}\, (t+1),\,\quad -1\leq t\leq 1
$$
where $g(t)=f(\sigma(t))\ .$

\medskip
\noindent
The monomials $1, t, t^2,\ldots$, are denoted by $\phi_j(t)= t^j$, for $j=0,1,\ldots$.

\begin{definition} (Degree $m$ of rule)\label{def2}

\noindent
Let $m\geq 0$ be an integer, $\mu_0=2$, $\mu_j=0$ for odd $j$ and  $\mu_j= \displaystyle \int_{-1}^1 \phi_j(t)\, dt=2/(j+1)$, for  even $j$, and let
$$
\mu_{m+1}=  \displaystyle \int_{-1}^1 \phi_{m+1}(t)\, dt.
$$
A rule $Q(g)$ is of degree $m\geq 0$ if it is exact for $\phi_j(t)$, with $0\leq j\leq m$, but not exact for  $\phi_{m+1}(t)$, that is,
$$
\left\{
\begin{array}{ll}
Q(\phi_0)&=\mu_0\\
Q(\phi_1)&=\mu_1\\
\qquad &\vdots\quad \\
Q(\phi_m)&= \mu_m,
\end{array}
\right.
$$
but
$$
Q(\phi_{m+1})\neq \mu_{m+1} \ .
$$
\end{definition}

\noindent
In what follows, $\mu_{m+1}$ (a.k.a. the principal moment) plays a key role and we shorten its notation to
\begin{equation}\label{pm}
\mu=\displaystyle  \int_{-1}^1\phi_{m+1}(t) \,dt =\displaystyle \frac{2}{m+2}\ .
\end{equation}

\begin{definition}\label{def3} (Sign of a rule)

\noindent
Let $Q_m(f)$ be a quadrature rule of degree $m \,(m\geq 0)$, such that 
$$
 \gamma_Q=\mu- Q(\phi_{m+1})\ .
$$
The rule is {\em positive} (resp. {\em negative}) if $\gamma_Q>0$ (resp. $\gamma_Q<0$).
\end{definition}

\noindent
A pair of rules of the same degree whose respective value $\gamma$ have opposite signs are rules of particular interest. Such rules will be called {\em companion} rules.

\begin{definition}\label{def4}(Companion rules)

\noindent
Two rules $A(g)$ and $B(g)$, of the same degree $m$, such that $\gamma_A$ and $\gamma_B$ have opposite signs are called {\em companion rules}.
\end{definition}

\section{Linear combination of two rules}\label{combination}

We now address the problem of combining a pair of companion rules and show that the resulting rule is not only a weighted mean of the two rules considered but also it has higher degree of precision. An obvious  advantage of this approach is that at a minor computational cost, the value of the new rule can be closer to the integral than the values of the rules of the starting pair (see Example \ref{exemplo2}).

\begin{proposition}\label{proposicao1}
Let $A(g)$ and $B(g)$ be two rules both of degree $m\geq 0$, such that
\begin{equation}\label{prop1A}
\mu_A=A(\phi_{m+1})\quad \neq\quad \mu_B=B(\phi_{m+1}),
\end{equation}
and consider the  linear combination of the rules
\begin{equation}\label{prop1B}
Y(g)=\displaystyle \frac{\mu-\mu_B}{\mu_A-\mu_B}\, A(g)+ \displaystyle \frac{\mu_A-\mu}{\mu_A-\mu_B}\, B(g), \,
\end{equation}
where $\mu$ is the principal moment in \eqref{pm}.
Then, the degree of the rule $Y(g)$ is at least $m+1$.
\end{proposition}
\begin{corollary}\label{cor1}
If $A(g)$ and $B(g)$ are companion rules of degree $m\geq 0$, such that
$$
\mu_B<\mu<\mu_A\quad \mbox{or}\quad \mu_A<\mu<\mu_B,
$$
then the combined rule \eqref{prop1B} has degree at least $m+1$ and
\begin{equation}\label{prop1D}
A(g)\leq Y(g)\leq B(g)\quad \mbox{or} \quad B(g)\leq Y(g)\leq A(g)\ .
\end{equation}
\end{corollary}

\begin{proof}[Proof of Proposition \ref{proposicao1}]

\noindent
Both rules are exact for $\phi_j(t)= t^j$, with $j=0,1,\ldots, m$, that is, $A(\phi_j)= B(\phi_j)= I(\phi_j)$. Therefore,
$$
\begin{array}{ll}
Y(\phi_j)&=\displaystyle \frac{\mu-\mu_B}{\mu_A-\mu_B}\, A(\phi_j)+ \displaystyle \frac{\mu_A-\mu}{\mu_A-\mu_B}\, B(\phi_j), \\
\\
& = \displaystyle \frac{\mu-\mu_B+\mu_A-\mu}{\mu_A-\mu_B}\,\, I(\phi_j)= I(\phi_j)\ .
\end{array}
$$
So, the rule $Y(g)$ has degree at least $m$. It remains to show  that this rule is exact for $\phi_{m+1}(t)$ which implies that its degree is at least $m+1$.
 From \eqref{prop1A}, we have
$$
\begin{array}{ll}
Y(\phi_{m+1})&=\displaystyle \frac{\mu-\mu_B}{\mu_A-\mu_B}\, \mu_A+ \displaystyle \frac{\mu_A-\mu}{\mu_A-\mu_B}\, \mu_B, \\
\\
& = \displaystyle \frac{\mu \mu_A-\mu\, \mu_B}{\mu_A-\mu_B}=\mu= I(\phi_{m+1})\ .
\end{array}
$$
\end{proof}
\begin{proof}[Proof of Corollary \ref{cor1}]

\noindent
We assume that $A(g)$ is positive and $B(g)$ is negative being the proof in the other case completely analogous. 
By definition of sign of a rule, we have $\mu>\mu_A$ and $\mu<\mu_B$, i.e.
$$
\mu_A<\mu<\mu_B\, .
$$
Let $\alpha=\mu-\mu_B<0$ and $\beta =\mu_A-\mu<0$ which implies  $\alpha+\beta =\mu_A-\mu_B<0$. The rule $Y(g)$ can be written as
$$
Y(g)= \displaystyle \frac{\alpha A(g)+\beta \, B(g)}{\alpha+\beta},
$$
where
$$
c_1=\displaystyle \frac{\alpha}{\alpha+\beta}>0\quad\mbox{and}\quad c_2=\displaystyle \frac{\beta}{\alpha+\beta}>0,
$$
So, the rule $Y(g)$ is a linear combination of the rules $A(g)$ and $B(g)$ with positive coefficients $c_1$ and $c_2$.  Consequently, the inequalities
in \eqref{prop1D} hold, and the rule $Y(g)$ is a {\em weighted mean} of the companion rules $A(g)$ and $B(g)$.
\end{proof}

\begin{example}{($S=2/3\, M+ 1/3\, T$)}\label{exemplosimpson}

\medskip
\noindent
The well-konown midpoint and trapezoidal rules are, respectively,
\begin{equation}\label{mid}
M(g)= 2\, g(0),\qquad
T(g)= g(-1)+ g(1).
\end{equation}
For $\phi_2(t)=t^2$, the principal moment is  
$$
\mu=\displaystyle \int_{-1}^1 \phi_2(t)\, dt= 2/3\ . 
$$
Since $M(\phi_2)=0$ and $T(\phi_2)=1$, we have
$$
\gamma_{M}= \mu -M(\phi_2)= 2/3>0\quad \mbox{and}\quad 
\gamma_T=\mu -T(\phi_2)= -1/3<0 \ .
$$
Thus, by Definition~\ref{def3}, $M(g)$ is positive, while $T(g)$ is negative. Moreover, as $\gamma_{M}\ . \gamma_{T}<0$,  the rules $M$ and $T$ are companion rules (see Definition~\ref{def4}).

\noindent
Let us show that the linear combination \eqref{prop1B} of these two rules of degree 1 coincides with the Simpson rule, which is a rule of degree $m=3$.

\medskip

\noindent
As $\mu_M= M(  \phi_2  )=0$ and $\mu_T=T(\phi_2)=2$, from \eqref{prop1B} it follows 

$$
\begin{array}{ll}
Y(g)&=\displaystyle{\frac{\mu-\mu_T}{\mu_M-\mu_T} M(g)+\frac{\mu_M-\mu}{\mu_M-\mu_T} T(g)}= \displaystyle \frac{2}{3}\, M(g)+ \displaystyle \frac{1}{3}\, T(g).
\end{array}
$$
Taking into account the expressions \eqref{mid}, we get
$$
Y(g)= \displaystyle \frac{1}{3}\left[ g(-1)+ 4\, g(0)+ g(1) \right],
$$
which coincides with Simpson rule $S(g)$ (see next example).
\end{example}

\begin{example} (A combined rule of degree 5 using the Simpson rule )\label{exemplo2}

\medskip
\noindent
One easily verifies  that in $[-1,1]$ the (open) rule
$$
A(g)=g\left(-\displaystyle \frac{\sqrt 3}{3}\right)+ g\left(\displaystyle \frac{\sqrt 3}{3}\right),
$$
and the Simpson rule (closed)
$$
S(g)=\displaystyle \frac{1}{3} \left( g(-1) + 4\, g(0)+ g(1)\right),
$$
are both of degree $3$. Let us show that they are companion rules, and verify that the respective combined rule \eqref{prop1B} has degree  $m=5$.

\noindent
 The principal moment is $\mu=\displaystyle \int_{-1}^1 t^4 \, dt =2/5$, and
$$
\begin{array}{l}
\mu_A=A(t^4)=2/9\quad \Longrightarrow\quad \gamma_A=\mu-A(\phi_4)=8/45>0   \\
\mu_S=S(t^4)=2/3   \quad \Longrightarrow \quad \gamma_S=\mu-S(\phi_4)=-4/15<0,   \\
\end{array}
$$
and so  $A(f)$ is positive while  $S(f)$ is negative. We have $\mu_A- \mu_S=2/9-2/3=-4/9$ and the rule $Y(g)$ has the form
\begin{equation}\label{nna}
Y(g)=\displaystyle \frac{\gamma_S}{\mu_A-\mu_S}\, A(g)- \displaystyle \frac{\gamma_A}{\mu_A-\mu_S}\, S(g)= \displaystyle \frac{3\, A(g)+ 2 \,S(g)}{5}.
\end{equation}
The explicit expression of the weighted mean \eqref{nna} is
\begin{equation}\label{nna1}
Y(g)=\displaystyle \frac{1}{15}\left[2\, g(-1)+ 9\, g(-\sqrt{3}/3)+ 8\, g(0)+ 9\, g(\sqrt{3}/3) + 2 \, g(1)\right] \ .
\end{equation}

\medskip
\noindent
The rule $Y(g)$ has degree $m=5$ since
$$
\begin{array}{l}
Y(\phi_4)=2/5\quad \mbox{and}\quad I(\phi_4)=2/5,\\
Y(\phi_5)=0\quad \mbox{and}\quad I(\phi_5)=0,\\
Y(\phi_6)=14/45\quad \neq \quad I(\phi_6)=2/7\ .\\
\end{array}
$$
The result  in \eqref{nna1}  shows the dependence of $Y(g)$  on 5 nodes. In paragraph \ref{3pontos} we will construct another rule of degree 5  using only 3 nodes.

\medskip
\noindent
In computational terms, given two rules $A(g)$ and $B(g)$  one does not need to use the expression of  the rule $Y(g)$ in terms of the nodes like in \eqref{nna1} but just the linear combination \eqref{nna}.  

\medskip
\noindent
In order to observe the numerical improvement one can get  passing from a pair of companion rules to the respective combined rule $Y(g)$, let us approximate the following integral which will be used as a test model in the subsequent examples:
$$
I(g)=\displaystyle \int_{-1}^1 \displaystyle \frac{2}{1+t^2} \,dt=\pi \ .
$$
For the first  rule we obtain
$A(g)=3$, for the Simpson rule $S(g)=10/3\simeq 3.333\ldots$ and for the combined rule $Y(g)= \left(3\, A(g)+ 2 \,S(g)\right)/5=47/15\simeq 3.1333\ldots$. Thus, the respective errors are
$$
E_A(g)=I(g)-A(g)\simeq 0.14,\quad E_S(g)=I(g)-S(g)\simeq -0.19,
$$
while $E_Y(g)=I(g)-Y(g)\simeq 0.008.$ Once computed $A(g)$ and $S(g)$ the cost for computing $Y(g)$ is only  an addition, two multiplications and a division. In this example an $\simeq 5 \verb+%+$  relative error  in $A(g)$ and $S(g)$ yields to a relative error of $\simeq 0.3 \verb+%+$ of the rule $Y(g)$, 
meaning that the combined rule is approximately 15 times more accurate than the two referred companion rules.
\end{example}

\subsection{Families of companion two\--point rules}

\noindent
Let $t_0, t_1$ be two distinct nodes belonging to  $[-1, 1]$, and  the two\--point quadrature rule
\begin{equation}\label{eq10}
Q(g)= A_0\, g(t_0)+ A_1 \, g(t_1),
\end{equation}
where the parameters $A_0$, $A_1$ will be determined in order  that the rule has a degree of precision at least 1.

\medskip
\noindent
Facts:

\begin{enumerate}
\item[(i)] There are infinite choices of points $(t_0,t_1)$ in the square region $D=[-1,1]\times [-1,1]$, for which  the rule $Q(g)$ (considered as a function of $(t_0,t_1)$)  is either positive or negative and  simultaneously has   degree 1;
\item[(ii)] There are more positive rules than negative ones;
\item[(iii)]  The points $(t_0,t_1)\in D$ for which  $Q(g)$ has exact degree $m=2$, lie on   the hyperbole ${\cal L}_2$,
 \begin{equation}\label{hip1}
{\cal L}_2:\qquad   1/3 +  \,t_0 \,t_1=0.
 \end{equation}
\item[(iv)] There is exactly one (two-point)  rule $Q(g)$ with degree $m=3$, namely for $t_0=-\sqrt{3}/3$,  $t_1=\sqrt{3}/3$,  that is given by
\begin{equation}\label{hip2}
Q(g)=g\left(   -\displaystyle \frac{\sqrt{3}}{3}\right)  + g\left( \displaystyle \frac{\sqrt{3}}{3} \right). 
\end{equation}
\end{enumerate}

   \begin{figure}[hbt] 
\begin{center}
  \includegraphics[scale=0.420]{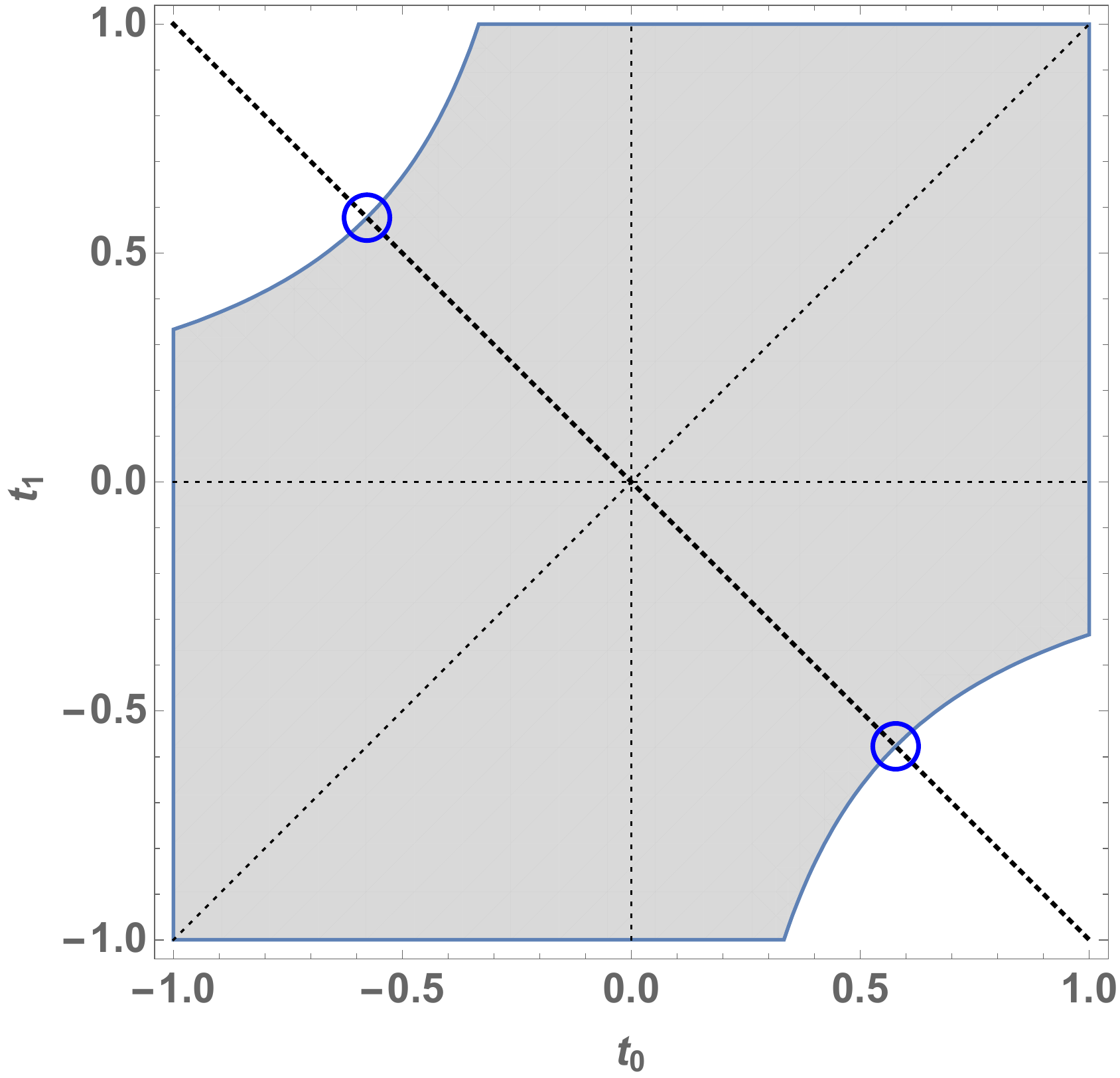} 
 \caption{Rules of degree 1: positive (gray) and negative (white). \label{fig1}}
\end{center}
\end{figure}

\medskip
\noindent
In Figure \ref{fig1} the points in the square region $D$ for which  the rule $Q(g)$ is  either positive or negative are  displayed in gray and white respectively.  The boundary of this region contains points of the hyperbole ${\cal L}_2$ in \eqref{hip1}.  For  $(t_0,t_1)$ belonging to the hyperbole  ${\cal L}_2$ the  rule $Q(g)$ has exactly degree  one, meaning that there is an infinite number of rules of $1$\--degree  each one with nodes $t_0, t_1$ for which  the point $(t_0,t_1)$ belongs to one of the two branches of the hyperbole  ${\cal L}_2$ in Figure \ref{fig1}. Moreover, one immediately sees that the two\--point closed Newton\--Cotes rule ($t_0=-1$ and $t_1=1$) is a negative rule of order one, as it was assumed in Example~\ref{exemplosimpson}. It is also worth  to recall that  any closed Newton\--Cotes rule of any degree $m>1$ is negative as well, whereas the open Newton-Cotes rules are all positive.

\medskip
\noindent
The mathematical aspects behind the facts $(i)$-$(iv)$ above and the geometry  of figures \ref{fig1}-\ref{fig2} can be explained as follows.

\noindent
Instead of the canonical polynomial basis of monomials of degree $\leq 3$, we consider the basis
$$
\begin{array}{l}
\Psi_0(t)=1\\
\Psi_1(t)=(t-t_0)\\
\Psi_2(t)=(t-t_0)\, (t-t_1)\quad \mbox{(nodal polynomial)}\\
\Psi_3(t)=(t-t_0)^2\, (t-t_1) \ . \\
\end{array}
$$
Imposing the condition that the rule $Q(f)$ is of degree at least 1, the parameters $A_0,A_1$ in \eqref{eq10} are computed. That is, considering
 $Q(\Psi_0)= I(\Psi_0)$ and 
$Q(\Psi_1)= I(\Psi_1)$, we obtain the following  triangular system in the unknowns $A_0,A_1$:
$$
\left\{
\begin{matrix}
A_0&+&A_1&= 2\\
  &&(t_1-t_0)\, A_1&= -2\, t_0,
  \end{matrix}
\right.
$$
   
   \begin{figure}[hbt] 
\begin{center}
  \includegraphics[scale=0.420]{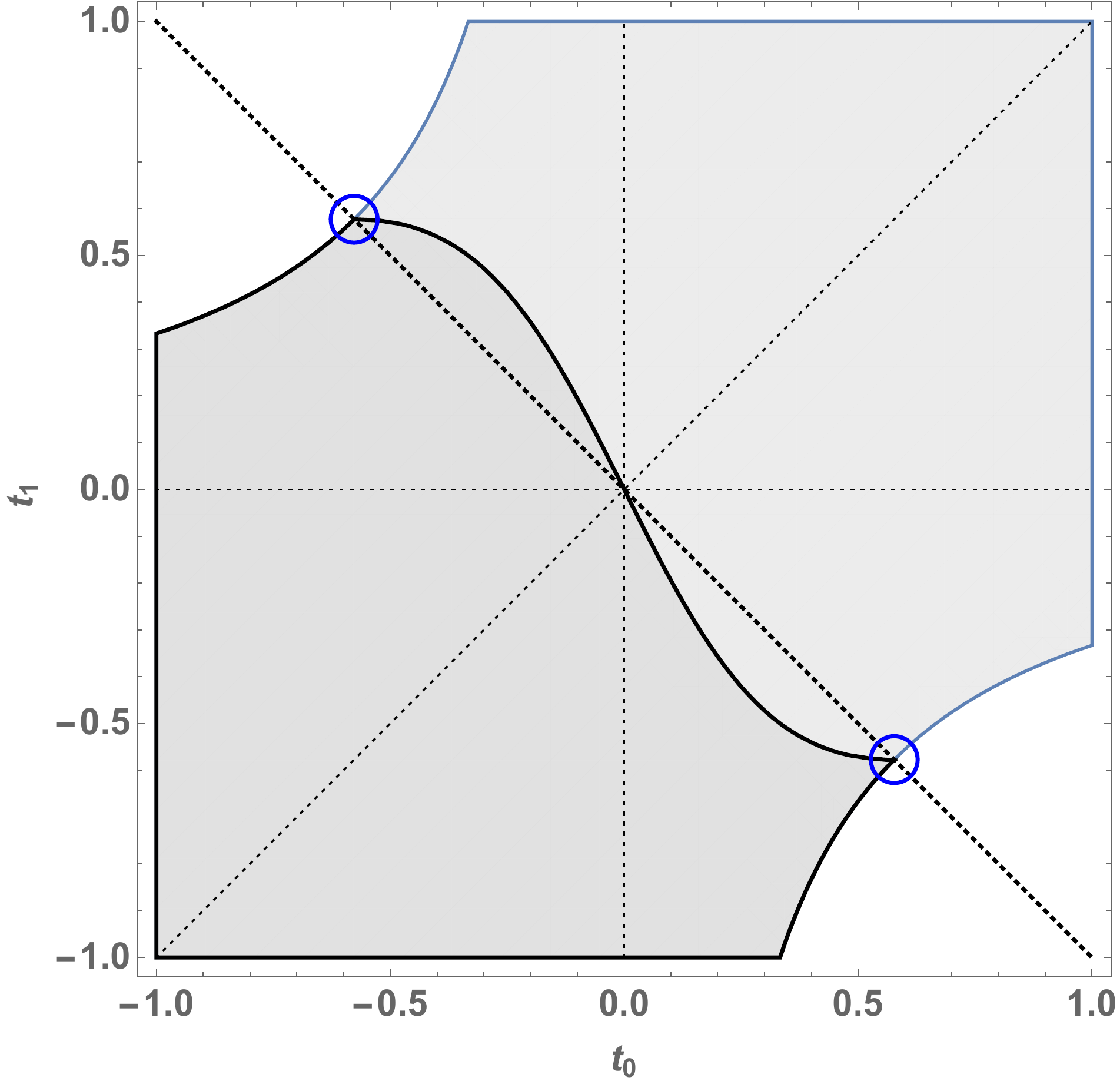} 
 \caption{Rules of degree $\geq 2$: positive (dark gray) and negative (light gray). \label{fig2}}
\end{center}
\end{figure}

\noindent
The solution of this system  is $A_1=-2\, t_0/(t_1-t_0)$ and $A_0= 2\, t_1/(t_1-t_0)$. Consequently, the two\--point rule
$$
Q(g)= \displaystyle \frac{2\, t_1}{t_1-t_0}\, g(t_0)-\displaystyle \frac{2\, t_0}{t_1-t_0}\, g(t_1), \qquad t_1\neq t_0
$$
has degree at least one. 

\medskip
\noindent
The rule applied to the nodal polynomial $\Psi_2(t)$ gives $Q(\Psi_2)=0$, whereas the moment   $I(\Psi_2)= 2/3 + 2 \,t_0\, t_1$. So, 
$$
 \gamma_Q=\mu- Q(\Psi_2) =\mu=\displaystyle \frac{2}{3}+ 2 \,t_0\, t_1\ .
 $$
Thus, the rule is of degree one and positive for the points $(t_0,t_1)$ displayed in gray color in Figure \ref{fig1} and negative for the white points.

\begin{remark}\label{rem1}
 One advantage of the polynomial basis $\Psi_0(t),\Psi_1(t), \ldots$ is that from the nodal polynomial onwards the rule is null and so the  value of the parameter $\gamma_Q$ coincides with $\mu=I(\Psi_3)$. This is the reason why we call $\mu$ the principal moment of the rule which also gives the  sign of a rule in the sense of Definition \ref{def3}.
\end{remark}

\medskip
 \noindent
The rule $Q(g)$ has degree $\geq 3$ if and only if \eqref{hip1} holds and it has null  moment $\mu=\displaystyle \int_{-1}^1 \Psi_3(t)\, dt $. The rule has degree 2 and is positive when $\mu>0$ and negative otherwise. The condition $\mu=0$ is given by the equation
\begin{equation}\label{pol4}
{\cal L}_3:\qquad  \displaystyle \frac{4\, t_0}{3}+ \displaystyle \frac{2\, t_1}{3} + 2\, t_0^2\, t_1 =0\ .
\end{equation}
 In Figure \ref{fig2}, the points  for which $\mu>0$ (positive rule of degree 2) are displayed in dark gray and in light gray the points  with  $\mu<0$ (negative rule). The boundary of the respective domains is shown in dark bold. 
This boundary is the algebraic curve  ${\cal L}_3$ defined by a cubic polynomial  as in \eqref{pol4}.

\medskip
\noindent
The intersection points of the curves ${\cal L}_2$ and ${\cal L}_3$ are symmetrically distributed with respect to the origin and are given by   $t_0=\pm \sqrt{3}/3$, with $t_1=-t_0$.  In  figures~\ref{fig1}-\ref{fig2} these intersection points are the center of the circles in blue.

\medskip
\noindent
 This fully justify the  facts $(i)$-$(iv)$.  Analogous procedures enable us to obtain  the geometry of a family of rules with  three distinct nodes.  Indeed, following the same lines as above, we will obtain a famous (positive) rule of degree 5, known as Gauss\--Legendre rule (see for instance \cite{brass}, Ch. 6) with three nodes, which might be be combined with the 5 degree rule $Y(g)$ given in Example \ref{exemplo2}.

\subsection{Families of companion three\--point rules}\label{3pontos}

In $D=\left\{(t_0,t_2,t_3):\,  -1\leq t_0\leq 1, \,  -1\leq t_1\leq 1,\, -1\leq t_2\leq 1 \right\}$, we consider the set of 3\--point quadrature rules
$$
Q(g)= A_0\, g(t_0)+ A_1\, g(t_1)+ A_2\, g(t_2),\quad\mbox{where}\quad t_0\neq t_1\neq t_2\ .
$$
Firstly, we obtain the weights $A_0, A_1, A_2$ in order that all the rules in the set have degree 2 and are either positive or negative. This will enable us to find candidates to pairs of companion rules of degree 2. 

\medskip
\noindent
The weights $A_i$ can be written as functions $A_i=\Omega_i: D\subset\Rb^3\mapsto \Rb$, for $i=0,1,2$. In fact, taking the polynomials
$$
\begin{array}{l}
\Psi_0(t)=1\\
\Psi_1(t)=t-t_0\\
\Psi_2(t)=(t-t_0)\, (t-t_1),
\end{array}
$$
 the rule has degree $\geq 2$ if and only if it is exact for $\Psi_0,\Psi_1$ and $\Psi_2$. That is, the weights are  solution to the triangular system
$$
\left\{
\begin{matrix}
A_0&+&A_1&+&A_2&=I(\Psi_0)\\
  &&(t_1-t_0)\, A_1&+& (t_2-t_0)\, A_2&= I(\Psi_1)\\
  &&  && (t_2-t_0)\, (t_2-t_1)\, A_2&=  I(\Psi_2).\\
  \end{matrix}
\right.
$$
Since we are assuming $t_0\neq t_1\neq t_2$, this system has a unique solution,
\begin{equation}\label{unique1}
\begin{array}{l}
A_2= \Omega_2(t_0,t_1,t_2)= \displaystyle \frac{ I(\Psi_2)}{(t_2-t_0)\, (t_2-t_1)}\\
\\
A_1= \Omega_1(t_0,t_1,t_2)= \displaystyle \frac{ I(\Psi_1)- (t_2-t_0)\, A_2}{(t_1-t_0)\, (t_2-t_1)}\\
\\
A_0= \Omega_0(t_0,t_1,t_2)= 2-(A_1-A_0)\ . \\
\end{array}
\end{equation}
Thus,
\begin{equation}\label{unique2}
Q(g)= \Omega_0(t_0,t_1,t_2)\, g(t_0)+\Omega_1(t_0,t_1,t_2)\, g(t_1)+\Omega_2(t_0,t_1,t_2)\, g(t_2),
\end{equation}
where $\Omega_i$ are given by \eqref{unique1}. After some simplifications the expressions in \eqref{unique1} become
\begin{equation}\label{unique1A}
\begin{array}{ll}
 \Omega_0(t_0,t_1,t_2)&= \displaystyle \frac{2\, (1+ 3\, t_1\, t_2)}{3\, (t_1-t_0)\, (t_2-t_0)}\\
 \\
  \Omega_1(t_0,t_1,t_2)&=- \displaystyle \frac{2\, (1+ 3\, t_0\, t_2)}{3\, (t_1-t_0)\, (t_2-t_1)}\\
  \\
    \Omega_2(t_0,t_1,t_2)&=\displaystyle \frac{2\, (1+ 3\, t_0\, t_1)}{3\, (t_2-t_0)\, (t_2-t_1)}\ .\\
  \end{array}
\end{equation}

\medskip
\noindent
Now, let us consider the polynomials
$$
\begin{array}{l}
\Psi_3(t)=(t-t_0)\, (t-t_1)\, (t-t_2)\quad \mbox{(nodal polynomial)}\\
\Psi_4(t)=\Psi_3(t)\, (t-t_0)\\
\Psi_5(t)=\Psi_4(t)\, (t-t_1).
\end{array}
$$

\begin{remark}\label{rem2}
\noindent
{\em 
Note that the rule \eqref{unique2} is null when applied to $\Psi_j$, for $j\geq 3$, and so the degree of the rule is the same of the first index $j\geq 3$ for which
$I(\Psi_j)\neq 0$. 
}
\end{remark}

   \begin{figure}[hbt] 
\begin{center}
  \includegraphics[scale=0.420]{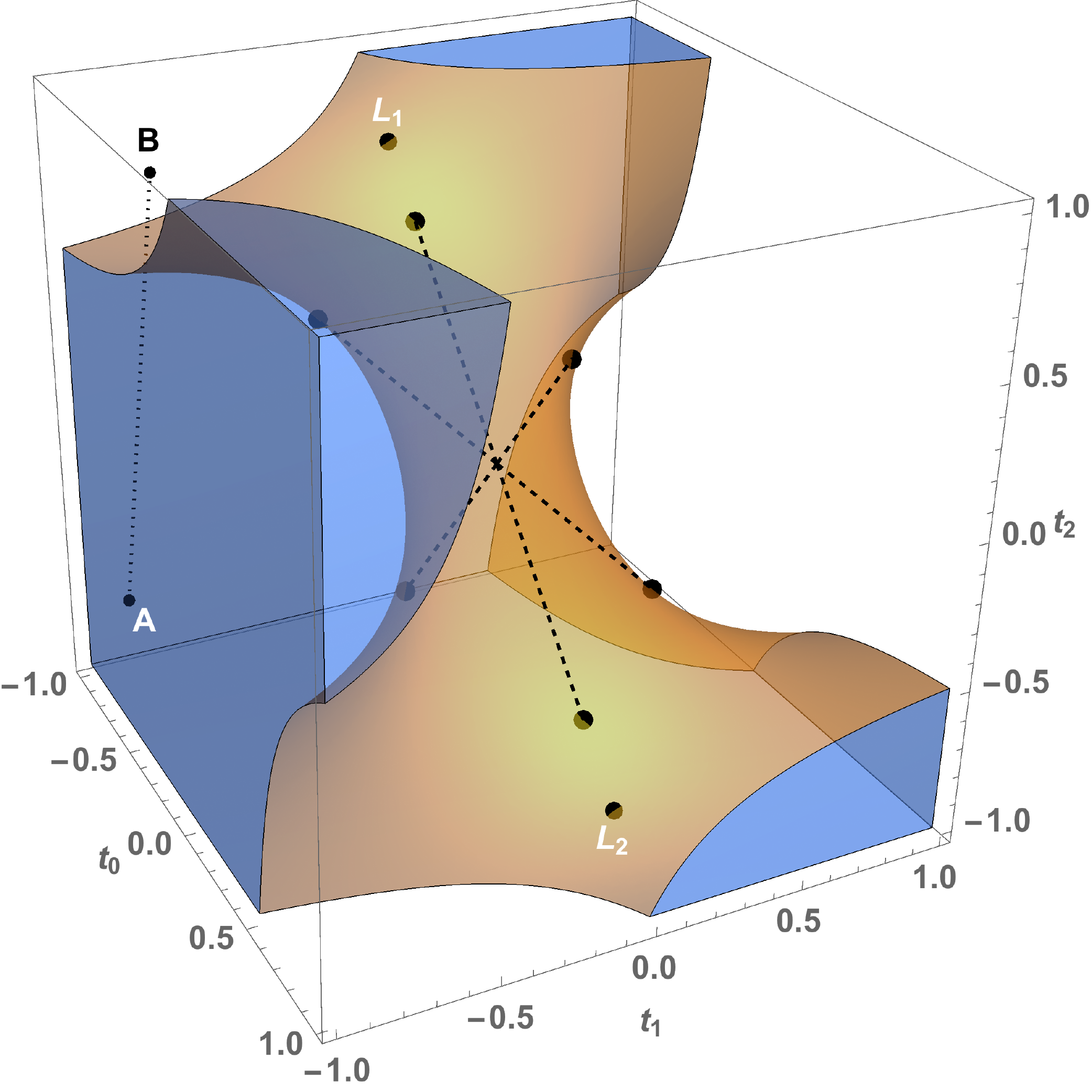} 
 \caption{Region  giving positive rules of degree 2. The symmetrical 6 points
 $\pm(0,\sqrt{3}/3, -\sqrt{3}/3)$ , $\pm(\sqrt{3}/3,0, -\sqrt{3}/3)$, $\pm(\sqrt{3}/3, -\sqrt{3}/3,0)$ belong to the boundary of the region. 
 For points A and B see Example \ref{exemplo3}. For the points $L_1$ and $L_2$   see Example \ref{exemplo5}.\label{fig3}}
\end{center}
\end{figure}

\medskip
\noindent
By construction, the rule $Q(g)$ has order $\geq 2$. Therefore, for any polynomial of the vector space ${\cal P}_2$ (polynomials of degree $\leq 2$) the rule is exact. In particular, it is exact for the elements of the canonical basis of ${\cal P}_2$, that is for $\phi_j(t)=t^j$, with $0\leq j\leq 2$. 
 
\noindent Since $I(\phi_3)=\displaystyle \int_{-1}^1 t^3\, dt =0$, we have $\gamma_3= - Q(\phi_3)$, where
$
\gamma_j=I(\phi_j)-Q(\phi_j).
$
Thus,  $Q(g)$ has order exactly 2 if $Q(\phi_3)\neq 0$ and order $\geq 3$ if $Q(\phi_3)=0$.

\medskip
\noindent
Let $P$ be the following polynomial 
\begin{equation}\label{unique3}
P(t_0,t_1,t_2)= t_0^3\, \Omega_0(t_0,t_1,t_2)+ t_1^3\, \Omega_1(t_0,t_1,t_2)+ t_2^3\, \Omega_2(t_0,t_1,t_2),
\end{equation}
where $\Omega_i$ are given by \eqref{unique1A}.

\medskip
\noindent The rule \eqref{unique2} is:

\begin{enumerate}
\item[(i)] positive if $P(t_0,t_1,t_2)<0$;
\item[(ii)]
 negative if $P(t_0,t_1,t_2)>0$;
\item[(iii)] of degree $\geq 3$ if $(t_0,t_1,t_2)$ is a root of the polynomial equation $P(t_0,t_1,t_2)=0$. 
\end{enumerate}

\medskip
\noindent
Specific instances of the rule \eqref{unique2} with weights \eqref{unique1A} are given in the following example.
 
 \begin{example}\label{exemplo3}
 
 \noindent
 Let $t_0=-15/16, t_1=-7/8$, $t_2=-3/4$
  (see Figure \ref{fig3} where the point $A$ has coordinates\footnote{After defining the polynomial expression \eqref{unique3}, the coordinates of the point A have been found using the predicate $\{P(t_0,t_1,t_2)<0, -1\leq t_0\leq 1, -1\leq t_1\leq 1, -1\leq t_2\leq 1, t_0\neq t_1\neq t_2\}$ as argument to the {\sl Mathematica} \cite{wolfram1} command FindInstance.}  $(t_0,t_1,t_2)$) and $P$ as in \eqref{unique3}. Computing the $\Omega_i$'s, given by \eqref{unique1A} we get  $P(t_0,t_1,t_2)<0$. Therefore the rule
$$
 A(g)= \displaystyle \frac{2}{9}\left[  760\, g(-15/16)- 1194\, g(-7/8)+ 443\, g(-3/4)  \right],
 $$
 is positive, with degree 2.
 
 \medskip
\noindent
Considering now $t_0=-3/4, t_1=-7/8$ and $t_2=-3/4$   (point $B$ in Figure \ref{fig3}) we have $P(t_0,t_1,t_2)>0$. Thus, the rule
$$
 B(g)= \displaystyle \frac{1}{117}\left[  95\, g(3/4)- 264\, g(-7/8)+ 403\, g(-3/4)  \right],
 $$
 is negative, with degree 2.
Taking  $\phi_3(t)=t^3$, the coefficients \eqref{prop1B} of the combined rule are:
 $$
 \begin{array}{l}
 \mu=I(\phi_3)=0\\
 \mu_A=A(\phi_3)= -2257/786\quad \mbox{(as $\mu-\mu_A>0$ the rule $A(g)$ is positive)}\\
  \mu_B=B(\phi_3)= 2\,975/19\, 984\quad \mbox{(as $\mu-\mu_B<0$ the rule $B(g)$ is negative)}\ .
 \end{array}
 $$
 As  $A(g)$ and $B(g)$ have errors of opposite sign they are companion rules (see Definition~\ref{def4}). The combined rule is
$$
 Y(g)= \displaystyle \frac{2\, 975}{32\, 316}\, A(g)+ \displaystyle \frac{29\,341}{32\,316}\, B(g),\\
 $$
which coincides with the (open) 4\--point rule
$$
 Y(g)=\displaystyle \frac{4522000\, g(-15/16)- 7659522\, g(-7/8) + 3545421\, g(-3/4)+ 
 214415\, g(3/4)}{290844}\ .
$$
It can be verified that $Y(\phi_j)= Y(\phi_j)$, for $j=0,1,2,3$, and
  $$I(\phi_4)- Y(\phi_4)=-58075361/220610560<0\ .$$Thus the combined rule $Y(g)$ is negative and of degree $m=3$.
 
 \medskip
 \noindent
  As an exercise, the interested reader can verify that for any of the $6$ points referred in Figure \ref{fig3} the corresponding rule $Q(g)$ in \eqref{unique2} has degree   3. Thus, one may conclude that there exists an infinite set of rules of the referred type which are negative and of  degree 3. Consequently, as all rules belonging to the set
  of open Newton\--Cotes rules,  ${\cal N}{\cal C}$,  are positive, so they are good candidates to use in pairs of companion rules in order to obtain combined rules of arbitrary order.
   \end{example}

  \begin{example}{($3$\--point Gauss\--Legendre  rule of degree 5)}\label{exemplo5}
  
  \medskip
  \noindent
   A simple choice of nodes $t_0,t_1$ and $t_2$ for the weights \eqref{unique1A} of the rule \eqref{unique2} is $t_1=0$, and $t_2=-t_0$, giving
 \begin{equation}\label{g1}
Q(g)=\displaystyle \frac{1}{3\,  t_0^2}\, g(t_0)  +\displaystyle \frac{3\, t_0^2-1}{3\,  t_0^2}\, g(0)  +\displaystyle \frac{1}{3 \,t_0^2}   \,g(-t_0)\ .
 \end{equation}

 \begin{table}[htbp]
   \tabcolsep=0.10cm
   \hspace{4cm}
   $$
 \begin{array}{| c |   c | c| c|}\hline
j & Q(\phi_j) & I(\phi_j) & \gamma_j=I(\phi_j)- Q(\phi_j) \\
\hline
3&0&0&0\\
\hline
4& 2\, t_0^2/3& 2/5& 2/15\, \left(  3-5\, t_0^2 \right)\\
 \hline
 5&0&0&0\\
\hline
6&2\, t_0^4/3&2/7& 2/21\, \left(  3-7\, t_0^4 \right)\\
\hline
 \end{array}
$$
 \caption{Rule \eqref{unique2} for $t_1=0$ and  $t_2=-t_0$. \label{tabela1} }
 \end{table}

\noindent
For the polynomials $\phi_j(t)=t^j$, with $j\geq 3$,  Table \ref{tabela1} displays the errors $\gamma_j=I(\phi_j)-Q(\phi_j)$.  The rule has degree  $\geq 4$ if and only if $\gamma_4=0$, that is, for $t_0=\sqrt{3/5}$. In this case the rule in \eqref{g1} becomes
 \begin{equation}\label{g2}
Q(g)=\displaystyle \frac{5}{9} \, g\left(-  \displaystyle \sqrt{\frac{3}{5}}   \right) +\displaystyle \frac{8}{9}\, g(0)  +\displaystyle \frac{5}{9} \, g\left(  \displaystyle \sqrt{\frac{3}{5}}   \right)\ .
 \end{equation}
 Taking into account the values of $\gamma_5=0$ and $\gamma_6>0$, the rule  \eqref{g2} has degree 5 and is positive. This is the  Gauss\--Legendre rule with 3 nodes.  In Figure \ref{fig3}, 
 the points $L_1$ and $L_2$ correspond to the nodes $(t_0, t_1,t_2)=\pm ( -\sqrt{3/5},0, \sqrt{3/5})$ for which the rule has degree 5.

  \end{example}
\section{The mean transformation rule $W(g)$}\label{mean}
 
 \noindent
In what follows,  ${\cal Q}$ denotes the  set of rules with degree $m\geq 0$. In Proposition~\ref{proposicao1}, a new rule has been  assigned to a pair of companion rules belonging to ${\cal Q}$.  In order to  generalize this scheme, let us  define a transformation
$$
\begin{array}{ll}
W:& {\cal Q}\times {\cal Q} \rightarrow Q\\
&(A,B)\mapsto W(g),
\end{array}
$$
where $W(g)$ will enjoy  analogous  properties of the linear combination $Y(g)$ in  \eqref{prop1B},  in the sense that the rule $W(g)$ has degree greater than those of the arguments $A(g)$ and $B(g)$. The rule $W(g)$ will be called {\em mean rule}.

\medskip
\noindent
As before, for a given rule $Q(g)\in{\cal Q}$, we compute the quantities $\mu$ and $\mu_Q$, defined by
\begin{equation}\label{m1}
\mu=\int_{-1}^1 \phi_{m+1} (t)\, dt\quad \mbox{and} \quad \mu_Q= Q(\phi_{m+1}),
\end{equation}
where $\phi_j(t)=t^j, \,\, j=0,1,\ldots$.

\begin{definition}{(Mean rule)}\label{defmean}

\medskip
\noindent
Let  ${\cal Q}$ be the  set of rules of degree $m\geq 0$ and  $A(g), B(g)$  belonging to ${\cal Q}$.   The mean rule of $A(g)$ and $B(g)$ is
\begin{equation}\label{m2}
\begin{array}{ll}
W(g)& = \left\{
\begin{array}{ll}
\displaystyle \frac{A(g)+B(g)}{2},& \mbox{if}\quad \mu_A=\mu_B\\
\\
\displaystyle \frac{(m+2)\, \mu_B-2}{(m+2)\, (\mu_B-\mu_A)}\, A(g)+ \displaystyle \frac{2- (m+2)\, \mu_A}{(m+2)\, (\mu_B-\mu_A)}\, B(g), & \mbox{if}\quad \mu_A\neq \mu_B\\
\end{array}
\right.
\end{array}
\end{equation}
 where $\mu_A$ and $\mu_B$ are as in \eqref{m1}.
\end{definition}
 
 \begin{proposition}\label{prop2}
 The mean rule \eqref{m2} has degree at least $m+1$.
 \end{proposition}
 \begin{proof}
 As the rules $A(g)$ and $B(g)$ have degree $m\geq 0$, we know that $A(\phi_j)=B(\phi_j)=0$, for odd $j$, and
\begin{equation}\label{m3}
 A(\phi_j)= B(\phi_j)=\displaystyle \int_{-1}^1\phi_j(t)\, dt= \displaystyle \frac{2}{j+1}, \quad j=0,2,4,\ldots, m\ .
\end{equation}
In the case $\mu_A=\mu_B$, from \eqref{m3}, it follows
$$
W(\phi_j)=\displaystyle \frac{A(\phi_j)+B(\phi_j)}{2}=I (\phi_j), \quad \mbox{for}\quad j=0,1,\ldots,m,m+1 .
$$
So, $W(g)$ has degree $\geq m+1$.

\medskip
\noindent
When $\mu_B-\mu_A\neq 0$, let us show that there exists a unique pair $(\alpha, \beta)$ such that
 $$
 W(g)=\alpha\, A(g)+\beta\, B(g), \quad \alpha, \beta\in \Rb,
 $$
 has degree $\geq m+1$.
 
 \medskip
 \noindent
For $j=0,1,\ldots, m,m+1$,  substituting $g$ in $W(g)$ by each $\phi_j(t)$  and applying \eqref{m3}, we obtain the linear system
$$
\left\{
\begin{matrix}
\alpha&+&\beta&=1\\
\mu_A\, \alpha&+ &\mu_B\, \beta&= \mu \ .
\end{matrix}
\right.
$$
 Since $\mu_A\neq \mu_B$ this system has the unique solution
 $$
 \alpha=\displaystyle \frac{(m+2)\, \mu_B-2}{(m+2)\, (\mu_B-\mu_A)}, \quad \mbox{and} \quad \beta=\displaystyle \frac{2- (m+2)\, \mu_A}{(m+2)\, (\mu_B-\mu_A)}\ .
 $$
 Thus,  $W(g)$ has degree at least $m+1$.
 \end{proof}
 
 \begin{example}\label{exemplo6}

 \noindent
We now construct the mean rule  of  \eqref{hip2} and the Simpson's rule (both rules of degree 3) and show that the mean rule has degree 5. The starting rules will be denoted by $A(g)$ and $S(g)$, respectively:
$$
A(g)=g(-\sqrt{3}/3)+g(\sqrt{3}/3),\qquad S(g)=1/3\, \left(g(-1)+ 4\, g(0)+ g(1)\right)\ .
$$
We have,
$$
 \mu=I(\phi_4)=2/5, \quad \mu_A= A(\phi_4)=2/9, \quad \mu_S=S(\phi_4)=2/3,\quad\mbox{and}\quad \mu_S-\mu_A=4/9\ .
$$
Thus,
\begin{equation}\label{m7}
  \begin{array}{ll}
 W(g)&= \displaystyle \frac{5\, \mu_S-2}{5\, (\mu_S-\mu_A)} A(g)+ \displaystyle \frac{2- 5\, \mu_A}{5\, (\mu_S-\mu_A)}\, S(g) = \displaystyle \frac{3}{5} A(g)+  \displaystyle \frac{2}{5} \, S(g)\\
 \\
 &= 2/15\, g(-1)+ 3/5\, g(-\sqrt{3}/3)+ 8/15\, g(0)+ 3/5\, g(\sqrt{3}/3)+2/15\, g(1)\ .
 \end{array}
\end{equation}
As $W(\phi_5)=I(\phi_5)=0$ the rule has degree $m=5$. 

\medskip
\noindent
We note that  $W(g)$ is a companion rule of the positive Gauss-Legendre rule \eqref{g2} since  $\gamma_6=I(\phi_6)-W(\phi_6)=-8/315<0$.  We may also obtain the mean rule of  \eqref{g2} and \eqref{m7}. That is, 
$$
\begin{array}{l}
\widetilde{W}(g)=1/630\, \left[54 \,g(-1) + 416\, g(0)+ 54 \,g(1)+ 125 \,g( -\sqrt{3/5} )+ \right.\\
  \hspace{3cm} +\left. 125\, g( \sqrt{3/5} )+ 243\, g(-\sqrt{3}/3) + 243 \,g(\sqrt{3}/3)\right]\ .
  \end{array}
  $$
This rule has degree $m=7$ and is negative since  $\gamma_8=I(\phi_8)-\widetilde{W}(\phi_8)= -16/1575$.

\medskip
\noindent
Let $g(t)=2/(1+t^2)$ and $I(g)=\int_{-1}^1 g(t)\, dt=\pi$. We have
$$
\widetilde{W}(g)=\displaystyle \frac{1321}{420}, \quad\mbox{whose error is}\quad I(g)-\widetilde{W}(g)\simeq -0.0036\ .
$$
Recall that the error with Simpson's rule is approximately $-0.19$ (see Example \ref{exemplo2}) and so the mean rule $\widetilde{W}$ leads to a remarkable gain in accuracy.
 \end{example}

 \subsection{The mean rule $W(g)$ as a least\--squares approximation}\label{best}
 
 \noindent
We now show that given two rules $A(g)$ and $B(g)$, its mean rule $W(g)$ is the least-squares approximation to the vector of moments
 \begin{equation}\label{mo1}
 \mathbf{h}=(\mu_0,\mu_1,\ldots,\mu_m, \mu)^T\,\, \in \Rb^{m+2},
 \end{equation}
 where, as before,  the moments are: $\mu_j=\displaystyle \int_{-1}^1 \phi_j(t)\, dt$  (for $j=0,\ldots, m$) and $\mu= \displaystyle \int_{-1}^1 \phi_{m+1}(t)\, dt$.
 
 \noindent
 We know that $A(\phi_j)=B(\phi_j)=\mu_i$, for $i=0,\ldots m$ and $A(\phi_{m+1})=\mu_A$,  $B(\phi_{m+1})=\mu_B$ are two distinct numbers (the case $\mu_A=\mu_B$ is trivial since the arithmetic mean is a least\--squares approximation of $h$ by rules of the type \eqref{mo2} below). 
 
 \medskip
 \noindent
 Consider the linear independent ($\mu_A\neq \mu_B$)  vectors of $ \Rb^{m+2}$:
 $$
 \begin{array}{l}
 \mathbf{v}_A= \left(A(\phi_0), A(\phi_1),\ldots, A(\phi_m), A(\phi_{m+1})\right)^T= (\mu_0,\mu_1,\ldots, \mu_m, \mu_A)^T\\
 \\
  \mathbf{v}_B= \left(B(\phi_0), B(\phi_1),\ldots, B(\phi_m), B(\phi_{m+1})\right)^T= (\mu_0,\mu_1,\ldots, \mu_m, \mu_B)^T,\\
 \end{array}
 $$
The least\--squares approximation of \eqref{mo1} by quadrature rules of the form
 \begin{equation}\label{mo2}
 Q(g)= \alpha \, A(g)+ \beta \, B(g), \quad \forall \alpha, \beta\in \Rb,
 \end{equation}
is equivalent to the least\--squares approximation of $\mathbf{h}$ by vectors of the form
 $$
 \mathbf{v}=\alpha \,  \mathbf{v}_A+ \beta\,  \mathbf{v}_B, \quad \forall \alpha, \beta\in \Rb.
 $$
That is,  the minimizer of the function
 $$
 F(\alpha,\beta)=\sum_{i=0}^{m+1} \left( \alpha\, v_{A,i}+ \beta \, v_{B,i}-h_i\right)^2,  \quad \forall \alpha, \beta\in \Rb\ .
 $$
 Denoting by $s$ the number $s=\sum_{i=0}^m \mu_i^2$, the minimum of $F$ is  the  solution of the system of normal equations
 $$
 \left[
 \begin{array}{cc}
 s+\mu_A^2& s+\mu_A\, \mu_B\\
  s+\mu_A\, \mu_B& s+ \mu_B^2
 \end{array}
 \right]\,
 \left[
 \begin{array}{l}
 \alpha\\
 \beta
 \end{array}
 \right] =
 \left[
 \begin{array}{l}
s+ \mu_A\, \mu\\
s+ \mu_B\, \mu\\
 \end{array}
 \right],
 $$
whose solution is
 $$
 \alpha= \displaystyle \frac{\mu-\mu_B}{\mu_A-\mu_B}, \quad \beta=\displaystyle \frac{\mu_A-\mu}{\mu_A-\mu_B} \ .
 $$
Thus,  the rule that is the best approximation of \eqref{mo1}, in the sense of least-squares, coincides with the mean rule \eqref{m2}. For other connections of quadrature with least\--squares approximations see \cite{graca}.

 \begin{example}(A mean rule of degree 7)\label{exemplo7}
 
 \medskip
 \noindent
 Consider the Gauss\--Legendre  rule \eqref{g2} 
$$
A(g)=\displaystyle \frac{5}{9} \, g\left(-  \displaystyle \sqrt{\frac{3}{5}}   \right) +\displaystyle \frac{8}{9}\, g(0)  +\displaystyle \frac{5}{9} \, g\left(  \displaystyle \sqrt{\frac{3}{5}}   \right)\ .
$$
and the open Newton\--Cotes rule with 5 nodes:
  $$
B(g)=\displaystyle \frac{1}{576} \, \left[275\, g(-4/5  ) +100\, g(-2/5)  +402\, g(0) + 100\, g(2/5) + 275\, g(4/5)  \right]\ .
$$
  \begin{table}[htbp]
   \tabcolsep=0.10cm
   \hspace{4cm}
   $$
 \begin{array}{| l |   c | }\hline
\qquad\qquad\mbox{Rule}& \mbox{Error} \\
\hline
A(g)=19/6& -0.0251\\
\hline
B(g)=3\,756/1\, 189& -0.0174\\
\hline
W(g)=156\, 637/49\, 938&0.00496\\
\hline
 \end{array}
$$
 \caption{Errors of $A(g)$, $B(g)$ and the mean rule of degree 7. \label{tabela2} }
 \end{table}
 The rules $A(g)$ and $B(g)$ have degree $m=5$ and are both positive with
 $$
 \begin{array}{l}
 \mu=\displaystyle \int_{-1}^1 \phi_6(t)\, dt =2/7,\\
 \\
 \mu_A=A(\phi_6)= 6/25\quad \Longrightarrow\quad \gamma_A= \mu-\mu_A= 8/175\simeq 0.046>0\\
 \\
  \mu_B=B(\phi_6)= 472/1875\quad \Longrightarrow\quad \gamma_B= \mu-\mu_B= 446/13125\simeq 0.034>0\ .
 \end{array}
 $$
 The respective mean rule of $A(g)$ and $B(g)$ is
\begin{equation}\label{m11}
 \begin{array}{l}
 W(g)=\displaystyle \frac{1}{11\, 088} \left[20\, 625\, g(-4/5)+7\, 500\, g(-2/5)+1\, 606\, g(0)+7\,500\,  g(2/5) + \right.\\
 \hspace{3cm} \left. + 20\, 625\,  g(4/5)-17\, 840\, g(-\sqrt{3/5})-17\, 840\,  g(\sqrt{3/5})\right] \ .
 \end{array}
\end{equation}
 This rule is positive of degree $m=7$:
 $$
 \begin{array}{l}
 W(\phi_j)= I(\phi_j), \,\, j=0,\ldots, 7,\\
  \mu_W=W(\phi_8)=26/125, \qquad \gamma_W= I(\phi_8)-\mu_W=16/1125\simeq 0.014>0 \ .
  \end{array}
 $$
 Note that $\gamma_W<  \gamma_B<\gamma_A$ and $\gamma_W/(\gamma_A+\gamma_B)/2)\simeq 0.36$, suggesting that the absolute error of $W(g)$  given by \eqref{m11}
 is approximately $1/3$ of the arithmetic mean of errors of the rules $A(g)$ and $B(g)$.
 
 \medskip
 \noindent
 For $I(g)=\displaystyle \int_{-1}^1 2/(1+t^2)\, dt=\pi$, we compare in Table \ref{figtab} the errors of $A(g)$, $B(g)$ with the error of the mean rule $W(g)$.
 
 \medskip
  \noindent
 Dividing the interval $[-1,1]$ into $n\geq 2$ equal parts, and considering  the composite rules $A_n(g)$, $B_n(g)$ and $W_n(g)$,   the gain of accuracy of the mean rule $W_n(g)$ relatively to the two rules of degree $5$ is numerically illustrated by the Table \ref{figtab}, where the respective errors are displayed. For subintervals of length $h=2/1024\simeq 0.0020$ the composite rule $W_{1024}(g)$ produces an approximation of $\pi$ with 33 significant digits.

   \begin{table}[hbt] 
\begin{center}
  \includegraphics[scale=0.420]{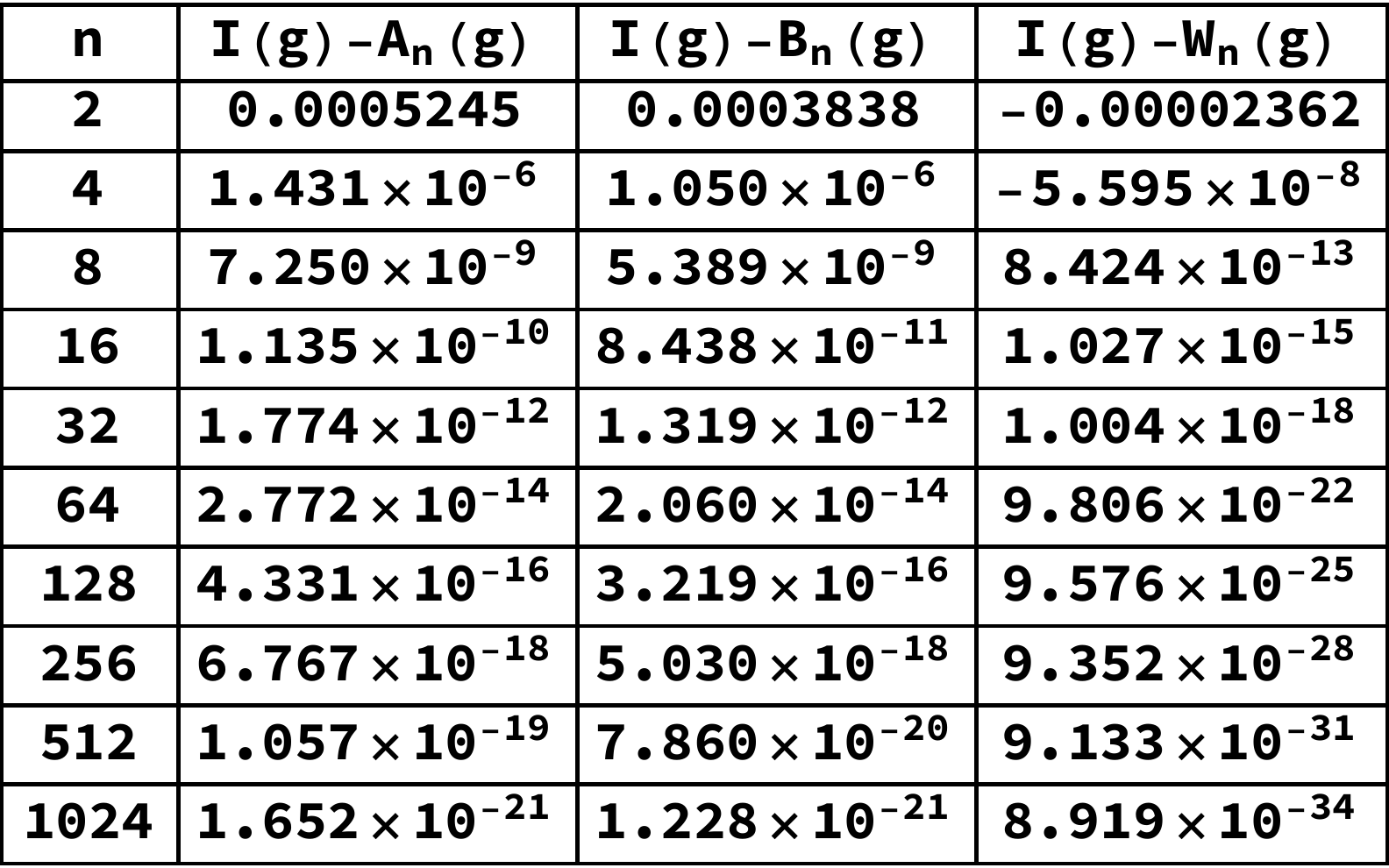} 
 \caption{Errors for the composite rules $A_n$, $B_n$ and $W_n$ . \label{figtab}}
\end{center}
\end{table}

  \end{example}

\section{Open rules of arbitrary degree with pseudorandom rational nodes}\label{openrules}

\noindent
The usual approach to approximate $I(g)=\displaystyle \int_{-1}^1 g/t)\, dt$ is to construct a quadrature rule  by mean of the interpolating polynomial of a given set of  nodes in $[-1,1]$. It is well-known that the resulting rules can be highly unstable when one increases the number of nodes. For instance, this is the case  of the closed Newton\--Cotes rules with a number of equally spaced nodes greater than  $ 10$. In order to overcoming such instability we are going to purpose the construction of rules of high degree by combinations of starting rules of degree one.

\medskip
\noindent
 For any integer $k\geq 1$, suppose it is given  $k+1$ {\em open} rules
\begin{equation}\label{open1}
\begin{array}{l}
Q_0(g)= 2\, g(0)\qquad\qquad  \mbox{(midpoint rule)}\\
Q_1(g)= g(-t_1)+ g(t_1) \\
 Q_2(g)= g(-t_2)+ g(t_2) \\
\hspace{2cm}\vdots  \\
Q_k(g)= g(-t_k)+ g(t_k), \\
\end{array}
\end{equation}
where  the symmetrical nodes $t_i$, for $i=1,\ldots,k$, belong to $(-1,1)$,  are nonzero distinct   {\em rational numbers}. Consider the combined rule
\begin{equation}\label{open2}
W_k(g)=a_0\, Q_0(g)+a_1\, Q_1(g)+\ldots+ a_k\, Q_k(g)\ .
\end{equation}

\noindent
In general the  rules \eqref{open1} are numerically stable and the assumption of the rationality of the nodes $t_i$, has the advantage of  obtaining values free of rounding errors provided exact computation is performed. This is possible with any symbolic language system whose arithmetic is exact for rational numbers such as  the {\sl Mathematica}. Thus, a rational linear combination of the rules \eqref{open1} will be numerically more interesting than the consideration of rules computed directly from an interpolatory rule of the $2\, k+1$ nodes.  

\medskip
\noindent
Due to the fact  we are considering nodes symmetrically distributed in  $(-1,1)$, for any odd  $j\geq 1$ we have
\begin{equation}\label{opena}
Q_i(\phi_j)= I(\phi_j)=0, \qquad \text{$j$ odd}, \quad  i=0,1,\ldots, k,
\end{equation}
 
 \medskip
 \noindent
  Moreover, all the rules in \eqref{open1} have degree one since
$$
\begin{array}{l}
Q_i(\phi_0)=2=I(\phi_0)\\
Q_i(\phi_1)=0=I(\phi_1)\\
Q_i(\phi_2)= 2\, t_i^2, \qquad\qquad i =0,\ldots, k\ .
\end{array}
$$
The rules \eqref{open1} cannot be of degree $2$, unless $t_i^2=1/3$, which is never the case since we are assuming the rationality of the nodes. If the combined rule \eqref{open2} has a certain even degree $d$, then $W_k(g)$ has degree at least $d+1$ due to \eqref{opena}. Thus, the degree of the combined rule is always odd.

\medskip
\noindent
In the following proposition we prove  that there exist unique {\em rational} coefficients $a_0,a_1,\ldots, a_k$ such that the linear combination  \eqref{open2}
has odd degree  $m=2\, k+1$. This result suggests that an efficient algorithm can be designed to obtain combined rules of arbitrary degree, starting from rules of degree one.

\begin{proposition}\label{prop3}
Let be given  $k+1$ rules as in  \eqref{open1} with $t_i\in\mathbb{Q}$, and consider the linear combination \eqref{open2} with coefficients $a_0, \ldots, a_k$. Then, 
\begin{enumerate}
\item[(i)] the weights $a_0,a_1,\ldots, a_k$ exist and are unique;
\item[(ii)] $a_0+a_1+\ldots+ a_k=1$;
\item[(iii)] the weights are rational numbers; 
\item[(iii)] the degree of the combined rule $W_k(g)$ in \eqref{open2} is $m=2\, k+1$.
\end{enumerate}
\end{proposition}

\begin{proof} For the sake of simplicity we just prove the statements for  $k=1$ and $k=2$ but the result for any   other $k\geq 1$ will follow by induction on $k$.

\noindent
 The combined rule for $k=1$ is
$$
W_1(g)=a_0\, Q_0(g)+a_1\, Q_1(g)\ .
$$
Due to \eqref{opena},  one has  $W_1(\phi_1)=0$ and $I(\phi_1)=0$. So, $W_1(g)$
has degree $d \geq 2$ if and only if the following two conditions hold:
$$
W_1(\phi)=I(\phi_i), \quad i=0, 2,
$$
 \noindent
that is,
$$
\left\{
\begin{array}{ccc}
2\,a_0&+2\, a_1&= 2\\
 & 2\, t_1^2\, a_1&=2/3\ .
\end{array}
\right.
$$
The first equation is  just $(ii)$. Moreover, the above system has the unique solution  $a_1=1/(3\, t_1^2), \,\, a_0=1-a_1$. Since $t_1\in\mathbb{Q}$, then  $a_0,a_1$ are also rational and $(i)$ and $(iii)$ hold. Finally, by
\eqref{opena} we have 
$$
W_1(\phi_3)=I(\phi_3)=0,
$$
which implies  that the rule has degree $m\geq 3$. However, as
$$
W_1(\phi_4)= 2/3\, t_1^2,  \quad I(\phi_4)=2/5,
$$
the equation $2/3\, t_1^2-2/5=0$ does not have solution in $\mathbb{Q}$
and so the rule $W_1(g)$ cannot have degree $4$. So, $m=3=2\, k+1$. 

\medskip
\noindent
Let $k=2$. Consider the combined rule 
$$
W_2(g)=a_0\, Q_0(g)+a_1\, Q_1(g)+a_2\, Q_2(g)\ .
$$
Due to \eqref{opena},  trivially  $W_1(\phi_1)=0$ and $I(\phi_3)=0$. So $W_2(g)$
has degree $d \geq 4$ if and only if the following three conditions hold:
$$
W_2(\phi)=I(\phi_i), \quad i=0, 2,4,
$$

\noindent
that is,
$$
\left\{
\begin{matrix}
2\,a_0&+&2\, a_1&+&2\, a_2&=& 2\\
 && 2\, t_1^2\, a_1&+ &2\,t_2^2\, a_2&=&2/3\\
  && 2\, t_1^4\, a_1& +&2\,t_2^4 \, a_2&=&2/5\ . 
\end{matrix}
\right.
$$
Equivalently,
\begin{equation}\label{open5}
 \left\{
\begin{matrix}
a_0&+& a_1&+&a_2&=& 1\\
 && t_1^2\, a_1&+& t_2^2 a_2 &=&1/3\\
  &&&& (t_1^2 \,t_2^2+ t_2^4) \,a_2&=&t_1^2/3+1/5\ .\\
\end{matrix}
\right.
 \end{equation}

\medskip
\noindent
The first equation of the above triangular system is just $(ii)$. Also,  as  the solution of the  this system consists of sums, products and quotients of nonzero rationals, then  $a_0,a_1,a_2\in \Qb$ and so $(i)$ is true. Finally, by
\eqref{opena} we have 
$$
W_2(\phi_5)=I(\phi_5)=0,
$$
and so that the rule has degree $m\geq 5$. However,
$$
W_2(\phi_6)\neq \quad I(\phi_6)=2/7 \ .
$$
which implies that  $W_2(g)$ cannot have degree $6$. So, its degree is  $m=5=2\, k+1$ and $(iii)$ holds. 
\end{proof}

\begin{remark}\label{rem3}
 
\noindent
{ \em The system \eqref{open5}  is almost singular if one takes the nodes $t_1\neq t_2$ very close to the central node $0$. Therefore for non exact arithmetic one expects that the combined rule $W_2(g)$ will be numerically unstable for $t_1,t_2 \simeq 0$. Thus, one might prefer a combined rule with distinct nodes $t_1,t_2$  closer to 1 rather than 0. However, for exact computations on the rationals the solution $a_0,a_1,a_2$ is exact and so, by construction, the rule $W_2(g)$ is stable, assuming that the function $g$ is sufficiently smooth in $[-1,1]$.
}
\end{remark}

\begin{example}{(A combined open rule of degree 7)}\label{exemplo8}

\medskip
\noindent
Consider the rules
$$
\begin{array}{l}
Q_0(g)= 2\, g(0)\\
Q_1(g)= g(-1/2)+ g(1/2)\\
Q_2(g)= g(-1/3)+ g(1/3)\\
Q_3(g)= g(-1/4)+ g(1/4),\\
\end{array}
$$
and
the respective combined rule 
$$
W_3(g)=a_0 \, Q_0(g)+a_1 \, Q_1(g)+a_2 \, Q_2(g)+a_3 \, Q_3(g)\ .
$$
As predicted by Proposition~\ref{prop3}, this rule has degree $m=2\,k+1=7$. Indeed,
the weights $a_i$ satisfy the conditions $W_3(\phi_i)= I(\phi_i)$, for $i=0,2,4,6=2\, k$, that is, they are solutions of the system
$$
\left[
\begin{array}{cccc}
1&1&1&1\\
0&1/2&2/9&1/8\\
0&1/8&2/81&1/128\\
0&1/32&2/729&1/2048
\end{array}
\right]
\,
\left[
\begin{array}{c}
a_0\\
a_1\\
a_2\\
a_3
\end{array}
\right]
 =
 \left[
\begin{array}{c}
1\\
2/3\\
2/5\\
2/7
\end{array}
\right]\ .
$$
Solving this system,  we get
\begin{equation}\label{open6}
\begin{array}{l}
W_3(g)= \displaystyle \frac{-4\,426}{105}\, Q_0(g)+ \displaystyle \frac{5\,344}{315}\, Q_1(g)
-\displaystyle \frac{5\, 589}{49}\, Q_2(g)+  \displaystyle \frac{309\, 248}{2\,205}\, Q_3(g).
\end{array}
\end{equation}
One may also confirm that the sum of the weights in \eqref{open6} is 1.
As
\begin{equation}\label{open7}
\gamma_{{}_{W_3}}= I(\phi_8)-W_3(\phi_8)=\displaystyle \frac{1\, 817}{15\, 120} \simeq 0.1202>0,
\end{equation}
the combined rule as degree $m=7$ (and is a positive rule).

\medskip
\noindent
In Example~\ref{exemplo7} we have constructed another rule of degree 7, the mean rule $W(g)$ in \eqref{m11}.
For the model function $g$, comparing the parameter $\gamma_{{}_{W_3}} \simeq 0.1202$  with the corresponding  parameter  $\gamma_W\simeq 0.014$,  it is expected that  the rule \eqref{m11} will perform better than the rule \eqref{open7} (particularly when we considerer the composite versions  as in Table \ref{figtab}). However, the  rule $W_3(g)$ uses only rational nodes whereas the rule \eqref{m11} does not. So, due to  Remark \ref{rem3}, we prefer the rule $W_3(g)$.

\end{example}

\medskip
\noindent
Composite rules of a certain degree $m$ for which the respective value of the parameter $\gamma$ is very close to zero are particularly useful. In fact,   in this case it means that the nodes $t_1,t_2,\ldots$ have been chosen such that the combined rule is almost optimal in the sense that it acts like a perturbed rule of the optimal rule of maximum degree. In the following example, from a set of starting rules of degree $1$ of type \eqref{open1}\--\eqref{open2}, we construct a $11$\--degree combined rule  whose parameter $\gamma$ is small.

\begin{example}{(A combined open rule of degree 11)}\label{exemplo9}

\medskip
\noindent
Consider the Legendre polynomial of degree 10
$$
P_{10}(t)=1/256 \left(-63 + 3465\, t^2 - 30030 \,t^4 + 90090 \,t^6 - 109395 \,t^8 + 
   46189\, t^{10}\right),
   $$
whose roots   belong to $(-1,1)$. It is well known that an interpolatory rule taking as nodes the 10 zeros of the polynomial $P_{10}(t)$  (know as a 10\--point Gauss\--Legendre rule) is a rule of maximal degree $m=17$.

\medskip
\noindent
A combined rule of the type \eqref{open2}, based on degree 1 rules \eqref{open1} whose nodes are (a certain number of) rational approximations of the zeros of $P_{10}$ is expected to be almost as accurate as the maximal degree rule of Gauss\--Legendre type. For instance,
consider for the positive nodes  the following rational approximations of the 5 positive roots of $P_{10}(t)$ (computed with an error\footnote{Using the {\sl Mathematica} command $Rationalize[argument, 10^{-16}]$ .} $\leq 10^{-16}$),
$$
\begin{array}{l}
t_1= \displaystyle \frac{41349881}{277750224},\,\,
 t_2=\displaystyle \frac{26322066}{60734531}\,\,
 t_3=\displaystyle\frac{ 209827923}{308838634}\,\,
t_4=\displaystyle \frac{130457471}{150806838}\,\,
t_5=\displaystyle \frac{ 272617463}{279921589} \ .
\end{array}
$$
These nodes define the rules  $Q_0(g), Q_1(g),\ldots, Q_5(g)$ in \eqref{open1}.
Let us now compute the weights of the corresponding combined rule
$$
W_5(g)=a_0\, Q_0(g)+a_1\, Q_1(g)+a_2\, Q_2(g)+a_3\, Q_3(g)+a_4\, Q_4(g)+a_5\, Q_5(g)\ .
$$
The conditions $W_5(\phi_i)= I(\phi_i),\,\, i=0,2,4,6,8,10$ give rise  to a system in the unknowns $a_0,\ldots,a_5$ whose solution is given in Figure \ref{figcoef}

  \begin{figure}[hbt] 
\begin{center}
\hspace{-3mm}
  \includegraphics[scale=0.53]{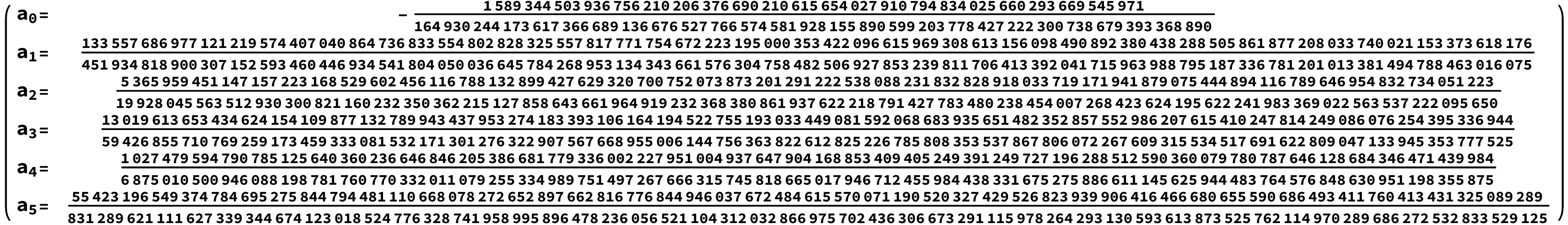} 
 \caption{Weights for the combined rule $W_5$ . \label{figcoef}}
\end{center}
\end{figure}

 \noindent
The rule has degree $m=11$ (see Proposition~\ref{prop3}) and its parameter  $\gamma_{{}_{W_5}}$ is
\begin{equation}\label{gama1}
 \gamma_{{}_{W_5(g)}}= I(\phi_{12})- W_5(\phi_{12})  \simeq 2.105*10^{-17}\ .
\end{equation}
Comparing with the parameter $\gamma_W\simeq 0.014 $ of the rule \eqref{m11}, in Example \ref{exemplo7},   (or  \eqref{open7} for a another rule of degree 7), one can predict that the composite rule $W_5(g)_n$, for $n$ subintervals of $[-1,1]$, will produce better numerical results than the rule $W$.

\noindent
For $I(g)= \displaystyle \int_{-1}^1 2/(1+t^2)\, dt=\pi$,  the Table~\ref{figtab4}  shows the errors of the composite rule $W_5 (g)_n$, for $n=2, 4,8,\ldots 1024$ subintervals. Comparing the values in this table with the ones in Table~\ref{figtab}  we observe a real improvement of the accuracy  relatively to the approximations obtained with the rule \eqref{m11}. Note that the last computed value of $W_5 (g)_{1024}$ gives an approximation of $\pi$ with $60$ significant digits. 
   \begin{table}[hbt] 
\begin{center}
  \includegraphics[scale=0.420]{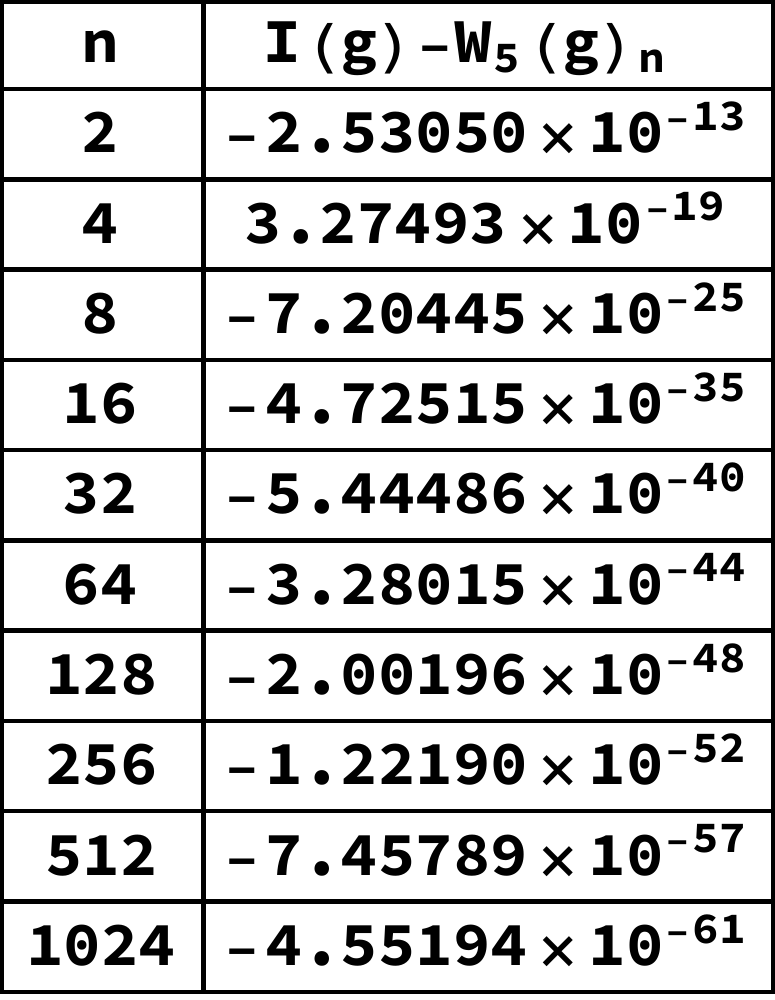} 
 \caption{Errors for the composite $W_5(g)_n$ rule for $g(t)=2/(1+t^2)$, with $-1\leq t\leq 1$  . \label{figtab4}}
\end{center}
\end{table}

\noindent
In spite of the starting rules $Q_0(g), \ldots Q_5(g)$ being rules of degree 1 only, we invite the reader to verify that a  rule such as the composite Simpson rule (degree 3) is totally unable to give an approximation of $\pi$ with the referred high precision of the value
$W_5 (g)_{1024}$.

\medskip
\noindent
Thus, we see that the rules $Q_0(g),\ldots,Q_5(g)$ are like a basis for the combined rules $W_5(g)$ of degree $m=11$. It is interesting to observe what happens when we replace the element $Q_0(g)$ (midpoint rule) by a new one, say $\tilde Q_0(g)=g(-1)+ g(1)$ (trapezoidal rule) and consider the new combined rule $\tilde W_5(g)$. The same code that produced \eqref{gama1} and the error values in the Table~\ref{figtab4} gives for $\tilde W_5(g)$:
$$
 \gamma_{{}_{\tilde W_5 (g)}}= I(\phi_{12})- \tilde W_5(\phi_{12})  \simeq - 5.243*10^{-18}\ .
 $$
and the respective errors are displayed in Table \ref{tabelafinal}.
 \begin{table}[hbt] 
\begin{center}
  \includegraphics[scale=0.420]{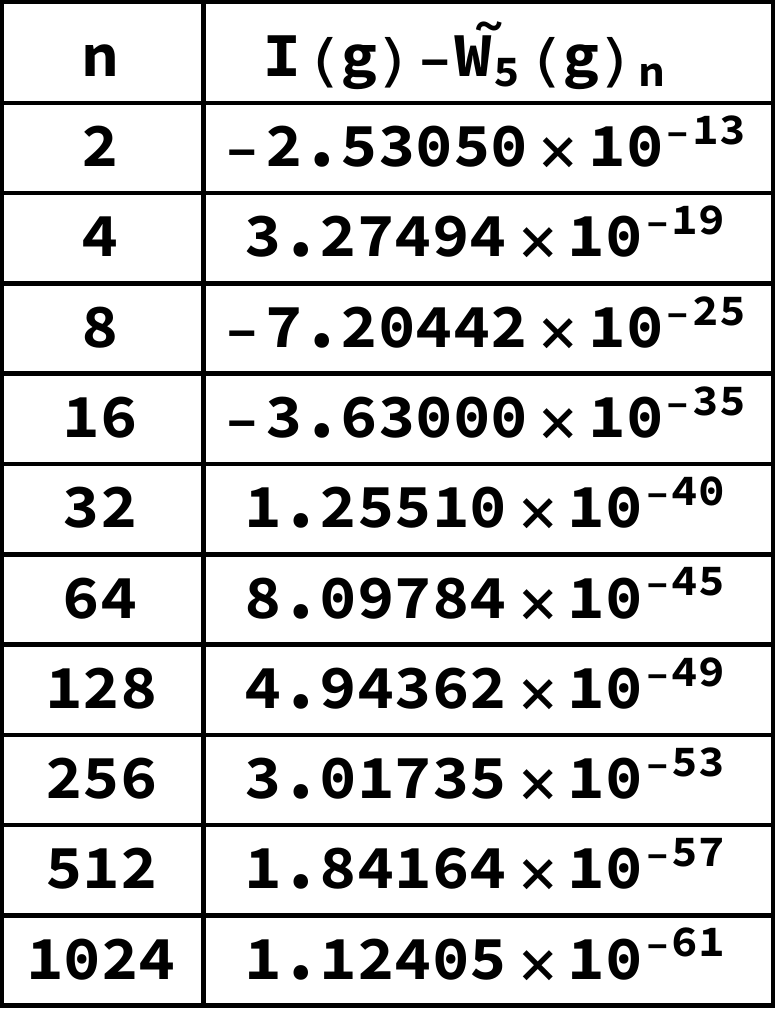} 
 \caption{With $Q_0(g)=g(-1)+ g(1)$, errors for the composite $\tilde W_5(g)_n$ (companion to the rule given in Table~\ref{figtab4}) . \label{tabelafinal}}
\end{center}
\end{table}

\medskip
\noindent
Thus $W_5 (g)$ and  $\tilde W_5 (g)$ are companion rules. Consequently, for $g(t)=2/(1+t^2)$, we have
$$
\tilde W_5 (g)_{1024}<\pi< W_5 (g)_{1024} \ . 
$$
In fact, rounding $\tilde W_5 (g)_{1024}$ to 61 decimal places we get
$$
\tilde W_5 (g)_{1024}=3.141592653589793238462643383279502884197169399375105820974944
$$
whose error is
$$
\pi-\tilde W_5 (g)_{1024}\simeq 1.12\,*\,10^{-61} \ .
$$
 \end{example}
\subsection{High degree pseudorandom combined rules}\label{secran}

Considering $k$ pseudorandom rational numbers in $(0,1)$, the  1\--degree starting rules $Q_i(g)$ and the respective combined rule $W_k(g)$
\begin{equation}\label{random1}
\begin{array}{ll}
Q_i(g)&= g(-t_i)+ g(t_i),\qquad i=0,\ldots k\\
W_k(g)&=\sum_{j=0}^k a_j\, Q_j(g),
\end{array}
\end{equation}
the Proposition~\ref{prop3} is obviously valid and so combined rules with pseudorandom nodes preserve the properties discussed before.

\begin{example}{(A pseudorandom rule of degree $m=151$)}\label{exemplo10}

\medskip
\noindent
In order to test the stability property of the combined rule when  $k$ is large, we  take $k=75$. Using the Mathematica command $\verb+SeedRandom[2020]+$ we generate 76 pseudorandom rational numbers\footnote{Recurring to the command $Rationalize[RandomReal[{0,1}],10^{-4}]$.}. Then, we compute the weights of the  combined rule $W_{75}(g)$  in \eqref{random1}. Finally,  we use  the composite version $W_{75}(g)_n$, for $n$ subintervals of $[-1,1]$,  applied to the function $g(t)=2/(1+t^2)$. 

\medskip
\noindent
The combined rule $W_{75}(g)$ has degree $m=151$ and its error parameter is
$$
 \gamma_{{}_{ W_{75} (g)}}= I(\phi_{152})-  W_{75}(\phi_{152})  \simeq 3.151*10^{-22}\ .
 $$
We show in Table~\ref{ola} the errors for the composite rule with  $n=2,4,8,\ldots, 1024$ subintervals, where $I(g)=\pi$. Thus, $W_{75}(g)_{1024}$ produces and approximation of $\pi$ with $507$ significant digits.

 \begin{table}[hbt] 
\begin{center}
  \includegraphics[scale=0.420]{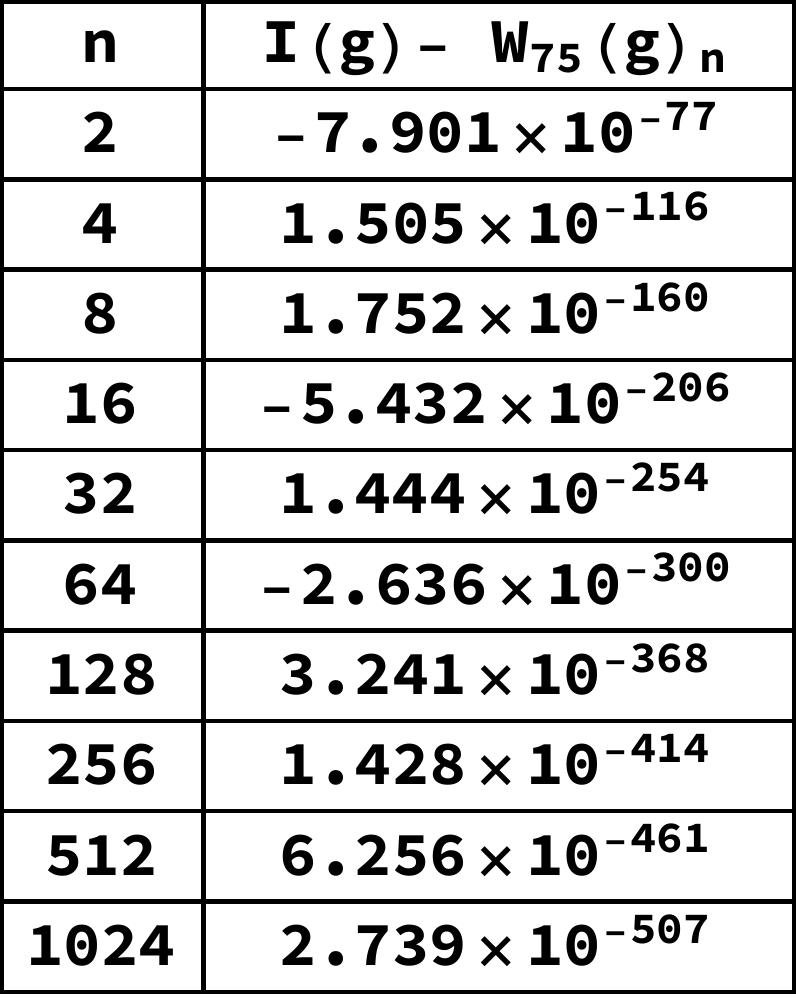} 
 \caption{ Errors for the composite rule $W_{75}(g)_n$  . \label{ola}}
\end{center}
\end{table}

\end{example}
 


\end{document}